\documentclass{gtart_a}
\pdfoutput=1

\usepackage{pinlabel,amscd}

\title{The integral cohomology of the group of loops}

\author[C~Jensen]{Craig Jensen}
\givenname{Craig}
\surname{Jensen}
\address{Department of Mathematics\\
         University of New Orleans\\\newline
         New Orleans, LA 70148\\USA}
\email{jensen@math.uno.edu}

\author[J~McCammond]{Jon McCammond} 
\givenname{Jon}
\surname{McCammond}
\address{Department of Mathematics\\ 
         University of California\\\newline 
         Santa Barbara, CA 93106\\USA}
\email{jon.mccammond@math.ucsb.edu}

\author[J~Meier]{John Meier}
\givenname{John}
\surname{Meier}
\address{Deptartment of Mathematics\\
         Lafayette College\\\newline
         Easton, PA 18042\\USA}
\email{meierj@lafayette.edu}

\volumenumber{10}
\issuenumber{}
\publicationyear{2006}
\papernumber{22}
\lognumber{0670}
\startpage{759}
\endpage{784}

\doi{}
\MR{}
\Zbl{}

\arxivreference{}  
\arxivpassword{}   


\received{16 October 2005}
\revised{6 February 2006}
\accepted{3 June 2006}
\proposed{Benson Farb}
\seconded{Joan Birman, Walter Neumann}
\published{11 July 2006}
\publishedonline{11 July 2006}
\corresponding{}
\editor{}
\version{}

\AtBeginDocument{\let\bar\wbar\let\hat\what}

\makeatletter
\def\cnewtheorem#1[#2]#3{\newtheorem{#1}{#3}[section]
\expandafter\let\csname c@#1\endcsname\c@thm}

\newtheorem*{mainthm}{Main Theorem}
\newtheorem*{blconj}{Corollary}
\newtheorem{thm}{Theorem}[section]
\cnewtheorem{lem}[thm]{Lemma}
\cnewtheorem{cor}[thm]{Corollary}

\cnewtheorem{prop}[thm]{Proposition}
\theoremstyle{definition}
\cnewtheorem{exmp}[thm]{Example}
\cnewtheorem{defn}[thm]{Definition}
\newtheorem{rem}{Remark}
\newtheorem*{ack}{Acknowledgments}

\makeatother  

\newcommand{\blackboard}[1]{\ensuremath{\mathbb{#1}}}
\newcommand{\script}[1]{\ensuremath{\mathcal{#1}}}
\newcommand{\smallcaps}[1]{\textrm{\textsc{#1}}}

\newcommand{\F}{\blackboard{F}}

\newcommand{\A}{\script{A}}
\newcommand{\B}{\script{B}}

\newcommand{\euler}{\raisebox{.5mm}{\ensuremath{\chi}}}

\newcommand{\psig}{{\mbox{\rm P}\Sigma}}

\newcommand{\opsig}{{\mbox{\rm OP}\Sigma}}
\newcommand{\balpha}{{\bar{\alpha}}}
\newcommand{\Aut}{\smallcaps{Aut}}
\newcommand{\Out}{\smallcaps{Out}}

\newcommand{\Stab}{\smallcaps{Stab}}

\newcommand{\MM}{{\mbox{MM}}}

\newcommand{\hyper}{\smallcaps{HT}}

\renewcommand{\rk}{{\mbox{\rm rk}}}

\begin{document}

\begin{asciiabstract}

\end{asciiabstract}

\begin{abstract} 
Let $\mbox{P}\Sigma_n$ denote the group that can be thought of either
as the group of motions of the trivial $n$--component link or the group
of symmetric automorphisms of a free group of rank $n$.  The integral
cohomology ring of $\mbox{P}\Sigma_n$ is determined, establishing a
conjecture of Brownstein and Lee.
\end{abstract}

\maketitle

\section{Introduction}

Let $L_n$ be a collection of $n$ unknotted, unlinked circles in
$3$--space, and let $\psig_n$ be the group of motions of $L_n$ where
each circle ends up back at its original position.  This group was
introduced in the PhD thesis of David Dahm, a student of Ralph Fox,
and was later studied by various authors, notably Deborah Goldsmith.
Alan Brownstein and Ronnie Lee succeeded in computing
$H^2(\psig_n,\Z)$ in \cite{BrLe93}, and at the end of their paper
conjecture a presentation for the algebra $H^*(\psig_n, \Z)$.  Further
evidence for this conjecture came when the cohomological dimension was
computed ($\mbox{cd}(\psig_n) = n-1$) by Collins in \cite{Co89}, and
when the Euler characteristic was computed ($\euler(\psig_n) =
(1-n)^{n-1}$) by the authors in \cite{McMe04} (see also
\cite{JeMcMeEuler}).  Here we establish the Brownstein--Lee
Conjecture.  As our argument is a mixture of spectral sequences and
combinatorial identities, it seems that Birman was quite prescient in
her \emph{Mathematical Review} of the Brownstein--Lee paper: ``The
combinatorics of the cohomology ring appears to be rich, and the
attendant geometric interpretations are very pleasing.''

Because $\pi_1(S^3 \setminus L_n) \cong \F_n$, it is not surprising
that there is a map
\[
\psig_n \rightarrow \Aut[\pi_1(S^3 
\setminus L_n)] \cong \Aut(\F_n)\ .
\]
Less immediate is that this map is injective, hence the group
$\psig_n$ can be represented as a group of free group automorphisms.
Its image in ${\Out}(\F_n)$ is denoted $\opsig_n$.  (This and other
background information is given in \fullref{sec:background}.)  The bulk
of our work focuses on the quotient $\opsig_n$.  We compute its
integral cohomology groups using the equivariant spectral sequence for
the action of $\opsig_n$ on a contractible complex introduced by
McCullough and Miller in \cite{McMi96}.  As is often the case, the
first page of this spectral sequence is charming yet opaque.
Combinatorial arguments are used to show that the $E_2$ page of the
spectral sequence is concentrated in a single column, hence one can
read off the cohomology groups from this page.  From this we get our
Main Theorem and then derive the Brownstein--Lee conjecture.

\begin{mainthm}
The Poincar\'e series for $H^*(\opsig_n, \Z)$---a formal power series
where the coefficient of $z^k$ is the rank of $H^k(\opsig_n,\Z)$---is
${\mathfrak p}(z) = (1 + nz)^{n-2}$.
\end{mainthm}

\begin{blconj}[The Brownstein--Lee Conjecture]
The cohomology of $H^*(\psig_n, \Z)$ is generated by one-dimensional
classes $\alpha_{ij}^*$ where $i \not = j$, subject to the relations:
\begin{enumerate}
\item $\alpha_{ij}^* \wedge \alpha_{ij}^* = 0$
\item $\alpha_{ij}^* \wedge \alpha_{ji}^* = 0$
\item $\alpha_{kj}^* \wedge \alpha_{ji}^* = (\alpha_{kj}^* -
  \alpha_{ij}^*) \wedge \alpha_{ki}^*$ 
\end{enumerate}
and the Poincar\'e series is ${\mathfrak p}(z) = (1 + nz)^{n-1}$.
\end{blconj}

\begin{rem}
In the late 1990s, Bogley and Krsti\'c constructed a $K(\psig_n, 1)$
whose universal cover embeds in Culler and Vogtmann's Outer Space.
Using this space they were able to establish the Brownstein--Lee
conjecture, but regrettably this work has not appeared.
\end{rem}

\begin{rem}
There are prior results on the (asymptotic) cohomology of $\psig_n$.
In Brady et al \cite{BMMM01} the cohomology of $\opsig_n$ and
$\psig_n$ with group ring coefficients are determined, and as a
corollary, it is shown that $\opsig_n$ and $\psig_n$ are duality
groups.  The $\ell^2$--cohomology is computed by the last two authors
in \cite{McMe04}, where it is shown that the $\ell^2$--cohomology of
$\opsig_n$ and $\psig_n$ is non-trivial and concentrated in top
dimension.  Alexandra Pettet has recently posted an article explaining
the complexity of the kernel of the natural ``forgetful" map $\psig_n
\twoheadrightarrow \psig_{n-1}$ \cite{Pe06}.
\end{rem}

\begin{ack}
We thank Ethan Berkove, Benson Farb, Allen Hatcher and Alexandra
Pettet for their interest and insights into this work.  We
particularly thank Fred Cohen for sharing some of his work with Jon
Pakianathan on an interesting subgroup of $\psig_n$, and for pointing
out a number of connections between this work and results in the
literature.  (See the end of \fullref{sec:cohopsig}.)
\end{ack}

\section{$\mbox{P}\Sigma_n$, $\mbox{OP}\Sigma_n$ and the complex $\MM_n$}
\label{sec:background}

This section moves at a brisk pace.  The reader completely unfamiliar
with the groups $\psig_n$ and the McCullough--Miller complexes $\MM_n$
should perhaps read the first six sections of \cite{McMe04} where the
material summarized in this section is developed in greater detail.

On the intuitive level, $\psig_n$ is the group of motions of $n$
unknotted, unlinked circles in the $3$--sphere.  In order to be more
precise and efficient, we present $\psig_n$ as a subgroup of the
automorphism group of a free group $\F_n$ (for background on this
isomorphism, see Goldsmith \cite{Go81}).  The group of \emph{pure
symmetric automorphisms} of $\F_n$ consists of all automorphisms that,
for a fixed basis $\{x_1, \ldots, x_n\}$, send each $x_i$ to a
conjugate of itself.  This group is generated by the automorphisms
$\alpha_{ij}$ induced by
\[
\alpha_{ij} = \left\{
\begin{array}{lc}
x_i \rightarrow x_j x_i x_j^{-1} & \cr
x_k \rightarrow x_k & k \ne i\,.\cr
\end{array}  \right.
\]
McCool proved that the relations
\[
\left\{
\begin{array}{ll}
[\alpha_{ij}, \alpha_{kl}] & i, j, k \mbox{ and } l\mbox{ all distinct}\cr
[\alpha_{ij}, \alpha_{kj}] &  i, j\mbox{ and }k\mbox{ distinct}\cr
[\alpha_{ij}, \alpha_{ik}\alpha_{jk}] & i, j\mbox{ and }k\mbox{ distinct}\cr
\end{array}
\right\}
\]
are sufficient to present $\psig_n$ \cite{McC86}.  Given $j \in [n]$
and $I \subset [n]\setminus \{j\}$ we let $\alpha_{Ij}$ denote the
product of generators
\[
\alpha_{Ij} = \prod_{i \in I} \alpha_{ij}\ .
\]
Using McCool's relations one sees that this product is independent of
the order in which one lists the $\alpha_{ij}$.  Note that when $I=
[n]\setminus \{j\}$, the element $\alpha_{Ij}$ is simply conjugation
by $x_j$.

For the remainder of this paper we view $\psig_n$ as a subgroup of
$\mbox{Aut}(\F_n)$.  Since the inner automorphisms,
$\mbox{Inn}(\F_n)$, form a subgroup of $\psig_n$, we may form the
quotient $\psig_n/\mbox{Inn}(\F_n)$, which we denote $\opsig_n$.
Interestingly, $\psig_n$ is a subgroup of the famous $\mbox{IA}_n =
\mbox{Ker}[\mbox{Aut}(\F_n) \twoheadrightarrow GL_n(\Z)]$.  In fact,
the set $\{\alpha_{ij}\}$ is a subset of the standard generating set
of $\mbox{IA}_n$ discovered by Magnus.  The image of $\psig_n$
contains the image of the pure braid group, $P_n$, under Artin's
embedding, hence there are proper inclusions
\[
P_n < \psig_n < \mbox{IA}_n < \mbox{Aut}(\F_n).
\]
Viewed as a group of free group automorphisms, there is an action of
$\psig_n$ on a contractible complex constructed by McCullough and
Miller \cite{McMi96}.  Our perspective on this complex is that of
\cite{McMe04} where it is described in terms of marked, $[n]$--labelled
hypertrees.  We quickly recall the definition of the $[n]$--labelled
hypertree poset.

\begin{defn}
A \emph{hypergraph} $\Gamma$ consists of a set of vertices $V$ and a
set of hyperedges $E$, each element of $E$ containing at least two
vertices.  We refer to edges in a hypergraph which contain more than
two vertices as \emph{fat} edges.  A \emph{walk} in a hypergraph
$\Gamma$ is a sequence $v_0,e_1,v_1,\ldots,v_{n-1},e_n,v_n$ where for
all $i$, $v_i \in V$, $e_i \in E$ and for each $e_i$,
$\{v_{i-1},v_i\}\subset e_i$.  A hypergraph is \emph{connected} if
every pair of vertices is joined by a walk.  A \emph{simple cycle} is
a walk that contains at least two edges, all of the $e_i$ are distinct
and all of the $v_i$ are distinct except $v_0=v_n$.  A hypergraph with
no simple cycles is a \emph{hyperforest} and a connected hyperforest
is a \emph{hypertree}.  Note that the no simple cycle condition
implies that distinct edges in $\Gamma$ have at most one vertex in
common

An \emph{$[n]$--labelled} hypertree is a hypertree whose vertices have
been labelled (bijectively) by $[n]=\{1, \ldots, n\}$.  Examples of
$[n]$--labelled hypertrees can be found in most figures in this paper.
The \emph{rank} of a hypertree $\tau$ is $\#E(\tau) - 1$.

For any fixed value of $n$, one can define a partial order on the set
of all $[n]$--labelled hypertrees: $\tau \le \tau'$ if every edge of
$\tau'$ is contained in an edge of $\tau$.  The poset consisting of
$[n]$--labelled hypertrees with this partial ordering is denoted
$\hyper_n$ and is called the \emph{hypertree poset}.  In
\fullref{fig:maxchain} we show a maximal chain in $\hyper_6$.
Notice that maximal elements in $\hyper_n$ correspond to ordinary,
that is simplicial, trees on $[n]$.
\end{defn}

\begin{figure}[ht!]
\begin{center}
  \includegraphics[width=4.5in]{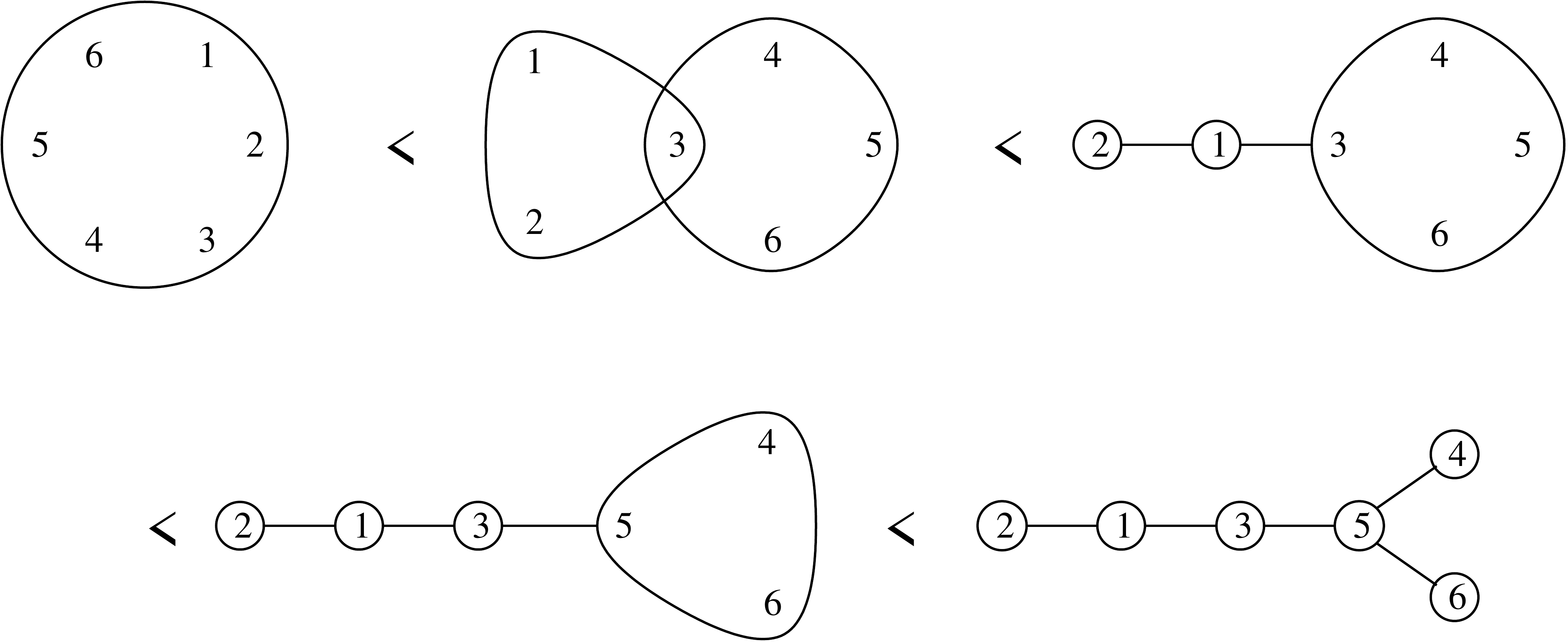}
 \caption{A maximal chain in $\hyper_6$}
 \label{fig:maxchain}
\end{center}
\end{figure}

The combinatorics of $\hyper_n$ are interesting (see \cite{BMMM01},
\cite{JeMcMeEuler} and \cite{McMe04}), but by and large we will need
few previously established combinatorial facts about this poset.  One
fact that we will use is:

\begin{lem}\label{lem:hyperhasmeets}
Given a collection of hypertrees $\{ \tau_1 , \ldots, \tau_k\}$ in
$\hyper_n$ there is a unique, maximal hypertree $\tau$ such that $\tau
\le \tau_i$ for each $i$.  In other words, $\hyper_n$ is a meet
semi-lattice.
\end{lem}

The McCullough--Miller complex $\MM_n$ was introduced in \cite{McMi96}
for the study of certain automorphism groups of free products.  The
main results we need in this paper are summarized in the following
Theorem.

\begin{thm}{\rm\cite{McMi96}}\label{thm:mc-miller-stuff}\qua
The space $\MM_n$ is a contractible simplicial complex admitting a
simplicial $\opsig_n$ action such that:
\begin{enumerate}
\item The fundamental domain is a strong fundamental domain and the
quotient $\opsig_n \backslash \MM_n$ is isomorphic to the geometric
realization of the poset $\hyper_n$.
\item If $\sigma$ is a simplex of $|\hyper_n|$ corresponding to a
chain $\tau_0 < \tau_1 < \cdots < \tau_k$, then
\[
\mbox{Stab}(\sigma) = \mbox{Stab}(\tau_0) \cong \Z^{\mbox{\scriptsize
    rk}(\tau_0)} 
\]
\item Let $j \in [n]$ and let $I$ be a subset of $[n]$ corresponding
to all the labels of a connected component of $\tau$ minus the vertex
labelled $j$.  Then the outer-automorphism given by the automorphism
$\alpha_{Ij}$ is contained in $\mbox{Stab}(\tau)$ and the collection
of all such outer-automorphisms is a generating set for
$\mbox{Stab}(\tau)$.
\end{enumerate}
\end{thm}

The generating set given above is not a minimal generating set.
Consider for example the maximal hypertree shown in
\fullref{fig:maxchain}.  If we let $\balpha$ denote the image of
$\alpha \in \psig_n$ in $\opsig_n$, then we have $\balpha_{45},
\balpha_{65}$ and $\balpha_{\{1,2,3\},5}$ are all in
$\mbox{Stab}(\tau)$.  But
\[
\balpha_{45} \cdot \balpha_{65} =
\left[\balpha_{\{1,2,3\},5}\right]^{-1} \mbox{ in }\opsig_6 
\]
as $\alpha_{45} \cdot  \alpha_{65} \cdot \alpha_{\{1,2,3\},5}$ is an
inner automorphism.

Using the equivariant spectral sequence, applied to the action
$\opsig_n \curvearrowright \MM_n$, we compute the cohomology groups
$H^i(\opsig_n, \Z)$.  Recall that the equivariant spectral sequence
for a group $G$ acting simplicially on a contractible complex is given
by
\[
E_1^{pq} = \prod_{\sigma \in {\mathcal E}_p} H^q (G_\sigma, M)
\Rightarrow H^{p + q}(G, M)
\]
where ${\mathcal E}_p$ denotes a set of representatives of the
$G$--orbits of $p$--cells.  The differentials on the $E_1$ page are the
standard ones of the equivariant cohomology spectral sequence, namely
a combination of restriction maps to a subgroup and coboundary maps.
(See \S VII.7 of \cite{Br94}.)  Since the simplex stabilizers for the
action of $\opsig_n \curvearrowright \MM_n$ are free abelian, and our
set of representatives of $\opsig_n$--orbits can be taken to be
$p$--simplices in the geometric realization $|\hyper_n|$, we can
exhibit the first page of the equivariant spectral sequence in a
fairly concrete manner.

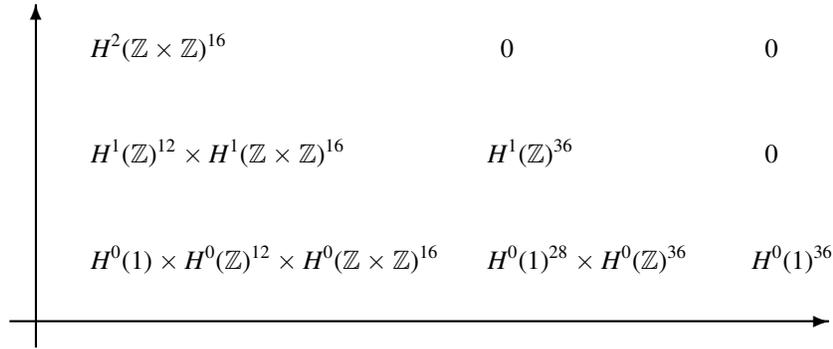
\begin{figure}[ht!]\small
\begin{center}
\begin{picture}(320,140)(-10, -10)
\thicklines
\put(-10,0){\vector(1,0){310}}
\put(0,-10){\vector(0,1){130}}
%
\put(20,20){$\displaystyle H^0(1) \times
	H^0(\Z)^{12} \times H^0(\Z \times \Z)^{16}$}
\put(20,60){$\displaystyle 
	H^1(\Z)^{12} \times H^1(\Z \times \Z)^{16}$}
\put(20,100){$\displaystyle  H^2(\Z \times \Z)^{16}$}

\put(170,20){$\displaystyle H^0(1)^{28} \times
	H^0(\Z)^{36}$}
\put(170,60){$\displaystyle 
	H^1(\Z)^{36}$}
\put(170,100){$\ \ 0$}
\put(270,20){$\displaystyle H^0(1)^{36}$}
\put(270,60){$\ \ 0$}
\put(270,100){$\ \ 0$}
\end{picture}
\end{center}
 \caption{The $E_1$ page of the equivariant spectral sequence for
the action  $\opsig_4 \curvearrowright \MM_4$}
 \label{fig:ssequiv}
 \end{figure}

For example, in \fullref{fig:ssequiv} we show the non-zero portion
of the $E_1$ page for the action of $\opsig_4$, where we have
suppressed the $\Z$--coefficients.  The geometric realization of
$\hyper_4$ was worked out as an example in \cite{McMi96}.  We redraw
their figure in \fullref{fig:HT4}.  The actual geometric
realization has dimension $2$, but the vertex corresponding to the
hypertree with exactly one edge forms a cone point in $|\hyper_4|$ and
so it is not shown in the figure.  The left edge of the $E_1$ page is
explained by the following observations: The $\opsig_4$ orbits of
vertices under the action $\opsig_4 \curvearrowright \MM_4$ correspond
to the elements of $\hyper_4$, and by direct observation one sees:
\begin{itemize}
\item There is one hypertree in $\hyper_4$ whose stabilizer is
trivial.  It is the hypertree with a single hyperedge, not shown in
\fullref{fig:HT4}.
\item There are twelve hypertrees whose stabilizers are $\cong \Z$.
These all have the same combinatorial type, shown in the top right of
\fullref{fig:HT4}.
\item The remaining sixteen hypertrees have stabilizers $\cong \Z
\times \Z$.
\end{itemize}
The interested reader may use \fullref{fig:HT4} to double-check our
entries in \fullref{fig:ssequiv}.

\begin{figure}[ht!]
\begin{center}
  \includegraphics[width=11.5cm]{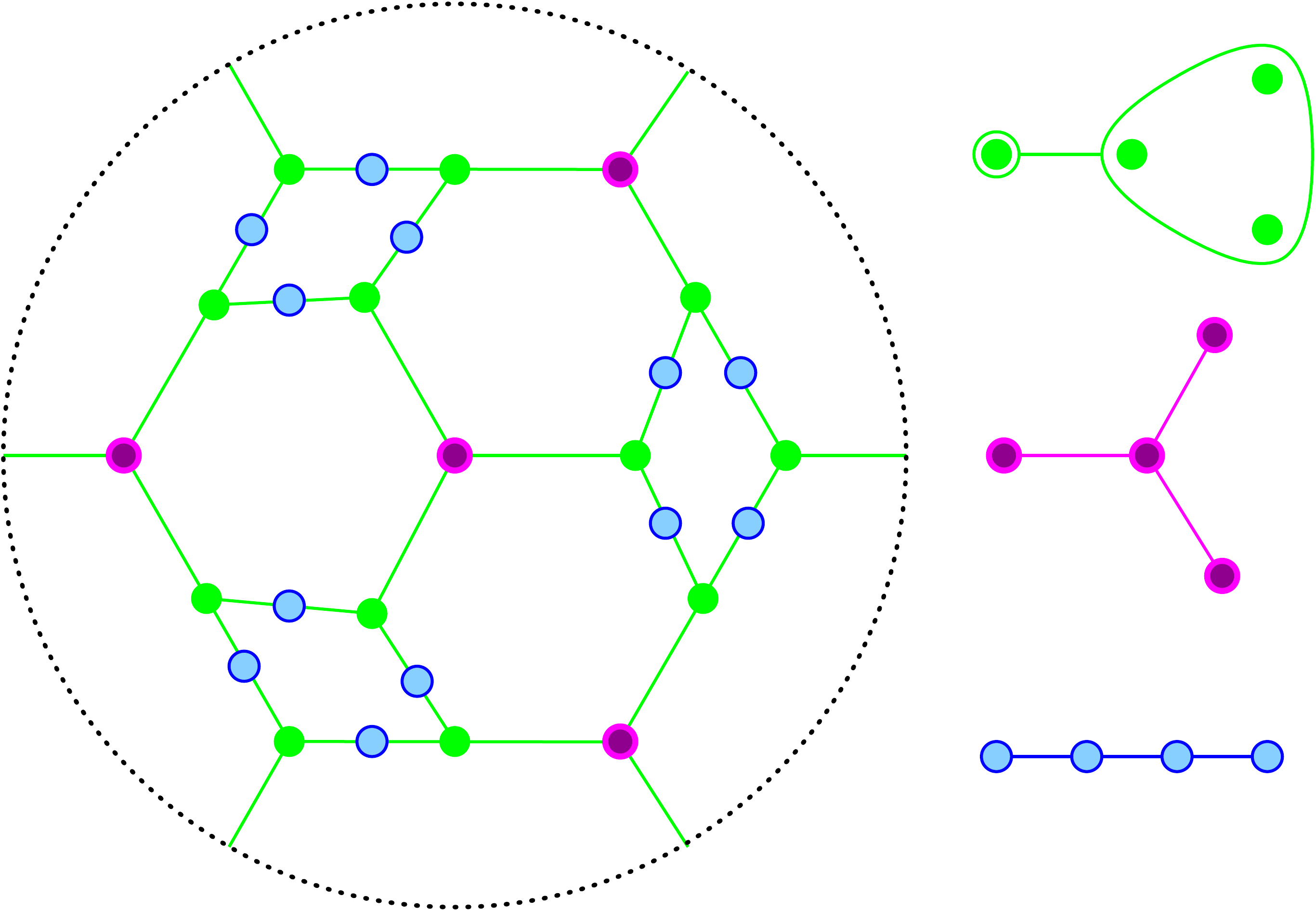}
 \caption{The geometric realization of $\hyper_4$, excluding the
 vertex corresponding to a hypertree with a single edge.  The graph is
 to be thought of as embedded in a projective plane, so edges leaving
 the dotted circle re-enter at the opposite point.  The color coding
 is by combinatorial type.}
 \label{fig:HT4}
\end{center}
\end{figure}

While we do give a concrete description of the $E_1$ page of this
spectral sequence, the reader should not be lulled into thinking that
in general this page is always directly accessible.  For example,
Cayley's formula states that the number of (simplicial) trees on $[n]$
is $n^{n-2}$.  This implies that the $(0,n-2)$--entry on the $E_1$ page
for the action $\opsig_n \curvearrowright \MM_n$ is
$H^{n-2}(\Z^{n-2})^{n^{n-2}}$ and this is most certainly not the
largest entry on the page.

\section{Choosing a basis}

We are able to understand the contents of the $E_1$ page of the
equivariant spectral sequence by finding a concrete, minimal set of
generators for each stabilizer of each simplex in $|\hyper_n|$.
Recall that the stabilizer of a simplex is the stabilizer of its
minimal vertex, hence it suffices to pick generators for the
stabilizers of individual hypertrees in the poset.  Further, the
stabilizer of a hypertree $\tau$ under the action of $\opsig_n$ is
free abelian of rank ${\rk(\tau)}$.  As an example, the stabilizer
under the action of $\opsig_6$ of the hypertree with two edges in
\fullref{fig:maxchain} is infinite cyclic.  It is generated by the
image of $\alpha_{\{1,2\},3}$ in $\opsig_6$, or equivalently by the
image of $\alpha_{\{4,5,6\},3}$.

For a fixed hypertree $\tau$ the chosen set of generators is
essentially described by the following conditions: You pick generators
as listed in \fullref{thm:mc-miller-stuff} except,
\begin{enumerate}
\item you never conjugate $x_1$, and 
\item you never conjugate $x_2$ by $x_1$.
\end{enumerate}
We now make this precise.

\begin{defn}
For each $j \in [n]$ let $\hat j$ be the index that is to be avoided
when conjugating by $x_j$.  That is, $\hat{j} = 1$ for $j \ne 1$ and
$\hat{1} = 2$.  Define a \emph{one--two automorphism} to be any
automorphism of the form $ \alpha_{Ij} = \prod_{i \in I} \alpha_{ij},
$ where $I \subset \{[n]\setminus\{j, \hat{j}\}\}$.  The image of
$\alpha_{Ij}$ in $\opsig_n$ is denoted $\balpha_{Ij}$ and is called a
\emph{one--two outer-automorphism}.

For a given hypertree $\tau$ the collection of all one--two
outer-automorphisms in $\mbox{Stab}(\tau)$ is the \emph{one--two basis}
for $\Stab(\tau)$.  We denote this by $\B(\tau)$.  If $\sigma$ is a
simplex in $|\hyper_n|$ corresponding to a chain $\tau_0 < \tau_1 <
\cdots < \tau_k$ then the stabilizer of $\sigma$ is the stabilizer of
$\tau_0$ and so we define $\B(\sigma) = \B(\tau_0)$.  A simplex
$\sigma$ in $|\hyper_n|$ is in the \emph{support} of a one--two
outer-automorphism $\balpha_{Ij}$ if $\balpha_{Ij} \in \B(\sigma)$.
\end{defn}

If one roots a hypertree $\tau$ at the vertex labelled $1$, then the
one--two basis for $\tau$ can be viewed as being (partly) induced by
``gravity".  The prescription to avoid conjugating $x_1$ by an $x_j$
corresponds to having $x_j$ conjugate the elements corresponding to
vertices below the vertex labelled $j$.  The exception is if the
vertex labelled $1$ has valence greater than one, in which case it
conjugates the elements below it, but not those in the branch
containing the vertex labelled by $2$.  Thus, the one--two basis for
the stabilizer of the hypertree shown in \fullref{fig:gravity} is
\[
\B(\tau) = \{\balpha_{\{4, 5, 8, 9, 10, 11\}, 1}, \balpha_{\{6, 7\},2}, 
\balpha_{\{8, 9\}, 4},\balpha_{\{10, 11\}, 5}\}\,.
\]
\begin{figure}[ht!]
\begin{center}
  \includegraphics[width=4.5in]{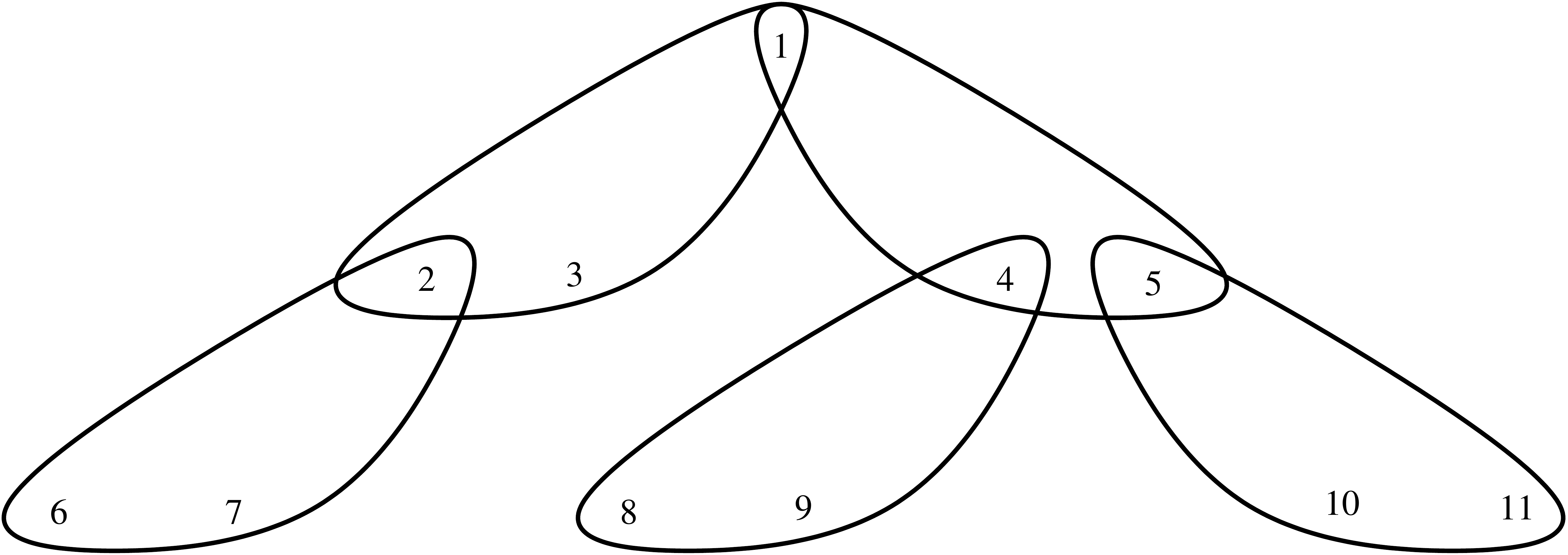}
 \caption{Rooting a hypertree at the vertex labelled $1$}
 \label{fig:gravity}
\end{center}
\end{figure}
The size of a one--two basis depends only on the rank of the underlying
hypertree, but the actual elements in the one--two basis very much
depend on the labelling.  Two rank-two, labelled hypertrees are shown
in \fullref{fig:rk2hyper}.  The underlying hypertrees are the same,
but the labels of the vertices are different.  Both have stabilizers
isomorphic to $\Z \times \Z$.  The one--two generating set for the
stabilizer of the labelled hypertree on the left is
$\{\balpha_{\{5,6\},1}, \balpha_{\{2,3\},4}\}$, while the generating
set for the stabilizer of the hypertree on the right is
$\{\balpha_{\{2,4,5\},3}, \balpha_{\{2, 4\},5} \}$.

\begin{figure}[ht!]
\begin{center}
  \includegraphics[width=12cm]{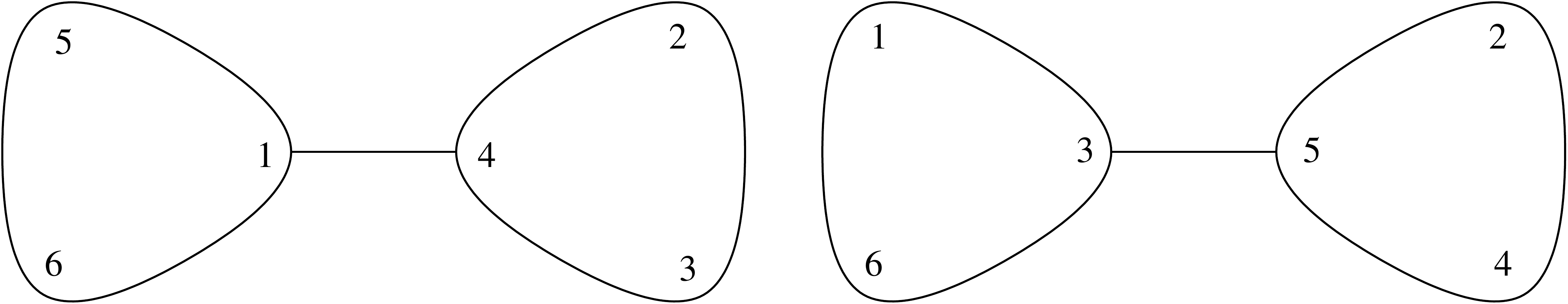}
 \caption{Two rank two labelled hypertrees in $\hyper_6$}
 \label{fig:rk2hyper}
\end{center}
\end{figure}

Finally, one last example: The stabilizer of the highest rank
hypertree in \fullref{fig:maxchain} is $\cong \Z^4$ and the one--two
generating set for $\Stab(\tau)$ is
\[
\B(\tau) = \{\balpha_{\{3, 4, 5, 6\},1}, \balpha_{\{4, 5, 6\},3},
\balpha_{4,5}, \balpha_{6,5}\}\ .
\] 
Proposition~5.1 of \cite{McMi96} gives a concrete description of the
stabilizers of vertices, from which one can derive:

\begin{lem}\label{lem:onetwobasis}
The set $\B(\tau)$ is a minimal rank generating set for
$\mbox{Stab}(\tau)$.
\end{lem}

(Hence our use of the term ``basis".)

A one--two outer-automorphism $\balpha_{Ij}$ may be in the stabilizer
of a simplex without being part of the one--two basis for the simplex.
Consider for example the hypertrees shown in \fullref{fig:nongood}.
The stabilizer of the rank one hypertree on the left is generated by
$\B(\tau) = \{\balpha_{\{5, 6, 7\}, 4}\}$.  However, if $\tau'$ is the
hypertree on the right, then $\B(\tau') = \{\balpha_{54},
\balpha_{\{6, 7\}, 4}, \balpha_{76}\}$.  Thus $\balpha_{\{5, 6, 7\},
4} \in \mbox{Stab}(\tau')$ but $\balpha_{\{5, 6, 7\}, 4} \not \in
\B(\tau')$.

\begin{figure}[ht!]
\begin{center}
  \includegraphics[width=12.0cm]{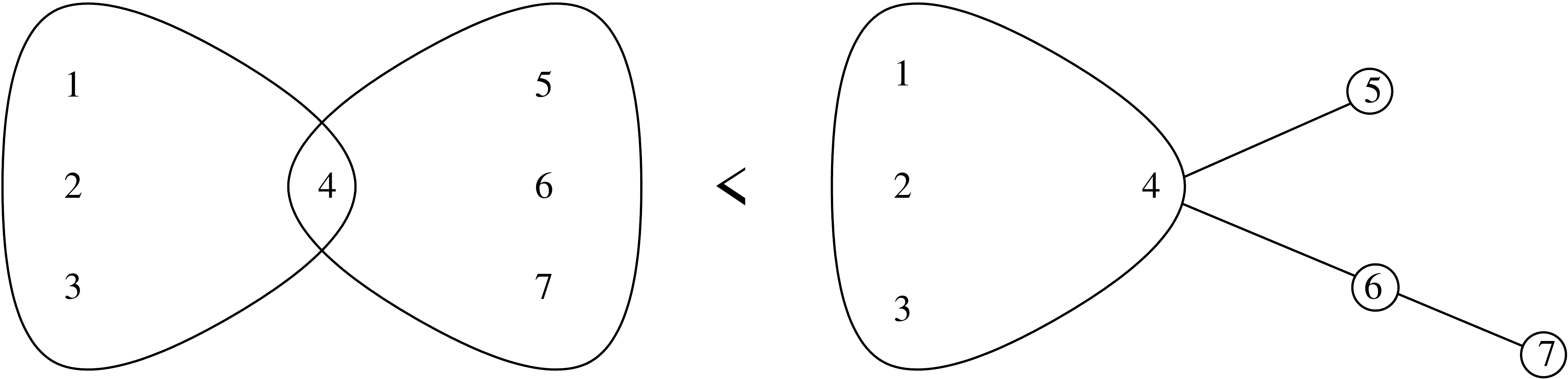}
 \caption{The outer-automorphism $\balpha_{\{5, 6, 7\}, 4}$ is
 contained in the stabilizer of both hypertrees.  It is in the one--two
 basis of the hypertree on the left, but it is not in the basis of the
 hypertree on the right.  (It is in the basis of the 1--simplex
 corresponding to the chain $\tau < \tau'$.)}
 \label{fig:nongood}
\end{center}
\end{figure}

The distinction between being in the stabilizer and being in the
one--two basis can be characterized combinatorially by removing
vertices.  Given any hypertree $\tau$ and a vertex $j$, we can
construct a new hypergraph on $[n]\setminus\{j\}$ by simply removing
$j$ from each hyperedge set and then removing any singleton sets that
result.  Let $\balpha_{Ij}$ be a one--two outer automorphism, let
$\tau$ be a hypertree and let $\tau'$ be the hyperforest obtained by
removing $j$ from the vertex set.  The stabilizer of $\tau$ contains
$\balpha_{Ij}$ if and only if $I$ is a union of vertices of connected
components of $\tau'$ and the one--two basis of $\tau$ contains
$\balpha_{Ij}$ if and only if $I$ is the vertex set of single
connected component of $\tau'$.

Here's the lovely fact.  Every simplex in $|\hyper_n|$ that supports a
one--two outer-automor\-phism $\balpha_{Ij}$ is compatible with the
hypertree consisting of two edges---$I \cup \{j\}$ and $[n] \setminus
I$---joined along the vertex labelled $j$.  Moreover, the one--two
basis for this two edge hypertree is $\{\balpha_{Ij}\}$.  (See
\fullref{fig:conepoint}.)  These facts extend to all subsets of
one--two bases (\fullref{lem:cone-prop}).

\begin{figure}[ht!]
\begin{center}
\labellist
\pinlabel $[n]\setminus\{j,I\}$ at 87 146
\pinlabel $j$ at 248 146
\pinlabel $I$ at 397 146
\endlabellist
\includegraphics[width=2.7in]{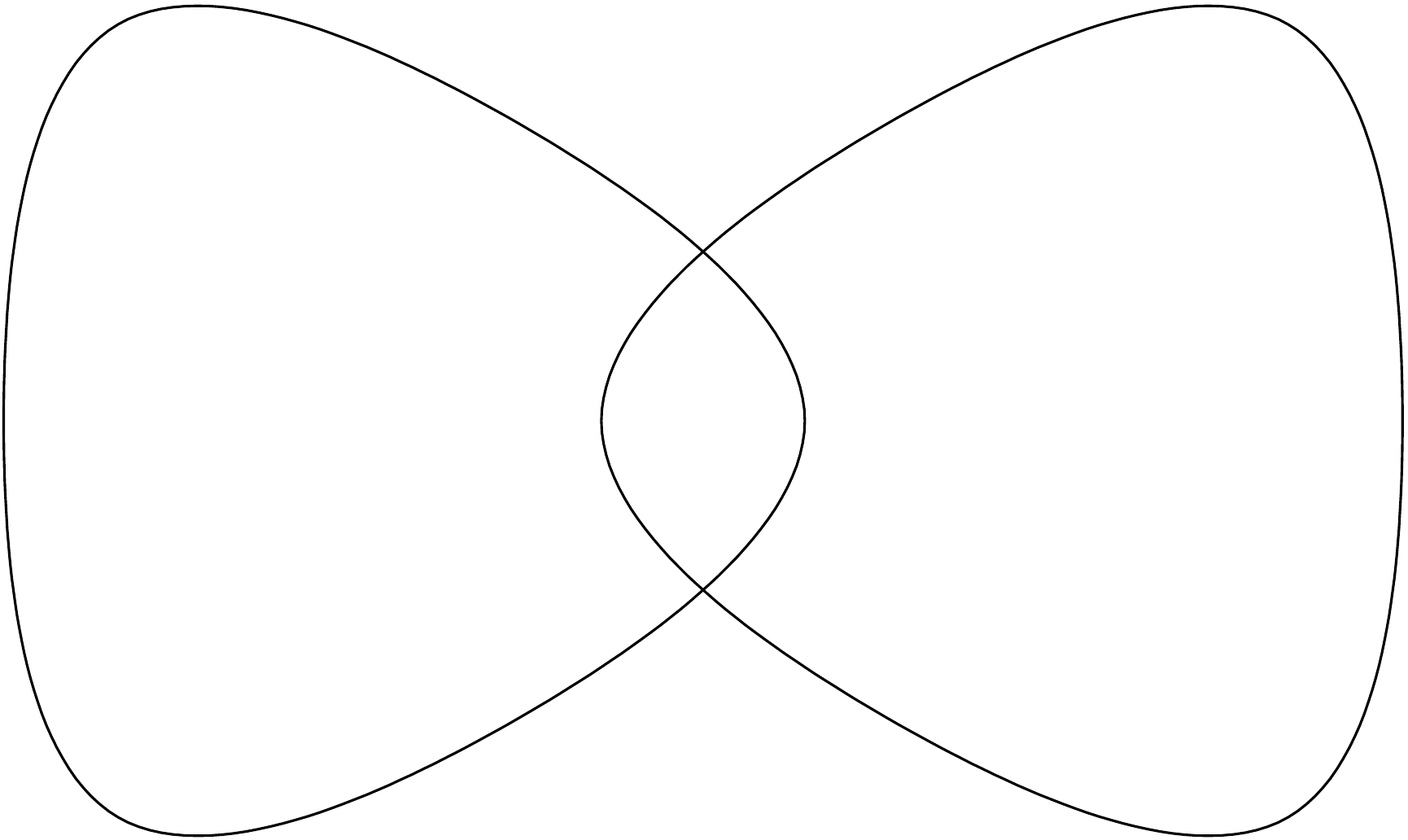}
 \caption{The cone point for a one--two basis element}
 \label{fig:conepoint}
\end{center}
\end{figure}

\begin{defn}
A collection $\A$ of one--two outer-automorphisms is called
\emph{compatible} if it is a subset of a one--two basis for some
hypertree $\tau$ ($\A \subset \B(\tau)$).  (The notion of a compatible
collection of one--two outer-automorphisms is essentially the same as
McCullough and Miller's notion of pairwise disjoint based partitions
of $[n]$ in \cite{McMi96}.)
\end{defn}

\begin{defn}
Let $\A$ be a compatible collection of one--two outer-automorphisms.
The collection of hypertrees that support $\A$ is the
\emph{$\A$--core}.  Call an element of the hypertree poset that
contains the set $\A$ in its stabilizer---but \underbar{does not
support it} in terms of being part of the one--two basis for the
hypertree---an \emph{$\A$--peripheral} hypertree.  Further, call the
subcomplex induced by the $\A$--core hypertrees the \emph{$\A$--core
complex}, and the subcomplex induced by $\A$--peripheral hypertrees the
\emph{$\A$--peripheral complex}.
\end{defn}

The fact that the hypertree poset has meets (\fullref{lem:hyperhasmeets})
implies the following:

\begin{lem}\label{lem:conepoint}
If $\A$ is a compatible collection of one--two outer-automorphisms then
there is a hypertree $\tau(\A)$, called the \emph{cone point of $\A$}
such that if $\A \subset \Stab(\tau')$ then $\tau(\A) \le \tau'$. 
\end{lem}

In fact, even more is true.  An easy induction on the number of edges
in $\tau$ (or the size of the compatible collection $\A$,
respectively) can be used to generalize the lovely facts listed above.

\begin{lem}\label{lem:cone-prop}
For all hypertrees $\tau$, the cone point of the one--two basis of
$\tau$ is $\tau$ itself ($\tau(\B(\tau)) = \tau$) and for all
compatible collections $\A$, the one--two basis of the cone point of
$\A$ is $\A$ itself ($\B(\tau(\A)) = \A$).  As a consequence, the cone
point hypertree $\tau(\A)$ lies in the $\A$--core.
\end{lem}

The combinatorics of the spectral sequence break into two cases.

\begin{defn}
A compatible collection of one--two outer-automorphisms $\A$ is
\emph{essential} if $\tau(\A)$ satisfies
\begin{enumerate}
\item There is at most one fat edge (an edge of order $>2$);
\item If there is a fat edge then the vertex labelled $1$ is in the fat edge;
\item If there is a fat edge then the unique reduced path from the
vertex labelled $1$ to the vertex labelled $2$ contains the fat edge.
\end{enumerate}
Notice that $\A$ is essential when $\tau(\A)$ is an ordinary tree.
Compatible collections of one--two outer-automorphisms that are not
essential are \emph{inessential}.
\end{defn}

For example, $\A= \{\balpha_{54}, \balpha_{\{6, 7\}, 4},
\balpha_{76}\}$ is an essential basis as it is the one--two basis for
the stabilizer of the rank three hypertree shown in
\fullref{fig:nongood}.

\begin{figure}[ht!]
\begin{center}
  \includegraphics[width=4.2cm]{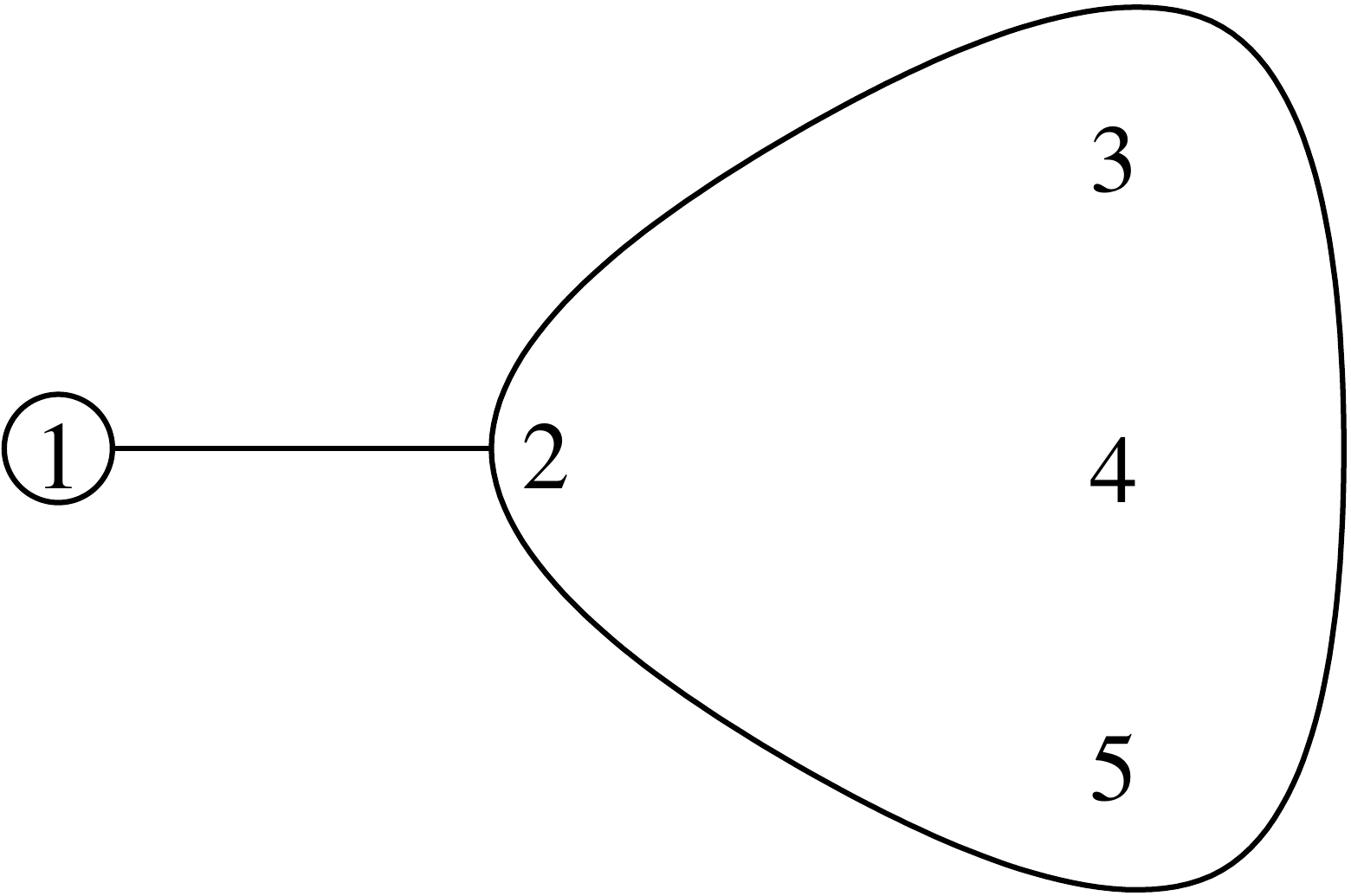}
 \caption{The hypertree $\tau(\{\balpha_{\{3,4,5\}, 2}\})$}
 \label{fig:examplecone}
\end{center}
\end{figure}
To build some intuition for what can occur with inessential one--two
bases, consider the one--two outer-automorphism $\balpha_{\{3,4,5\}, 2}
\in \opsig_5$. The cone point $\tau(\{\balpha_{\{3,4,5\}, 2}\})$ is
shown in \fullref{fig:examplecone}.
While there is a single fat edge, the fat edge does not contain the
vertex labelled by $1$, nor is it in the minimal path from the vertex
labelled $1$ to the vertex labelled $2$.  The associated
$\{\balpha_{\{3,4,5\}, 2}\}$--peripheral subcomplex of $|\hyper_5|$ is
shown in \fullref{fig:peripheral}.  The outer-automorphism
$\balpha_{\{3,4,5\},2}$ is in the stabilizer of every hypertree in
this subcomplex, but it is not in the one--two basis for any of these
hypertrees.

\begin{figure}[ht!]
\begin{center}
  \includegraphics[width=12.5cm]{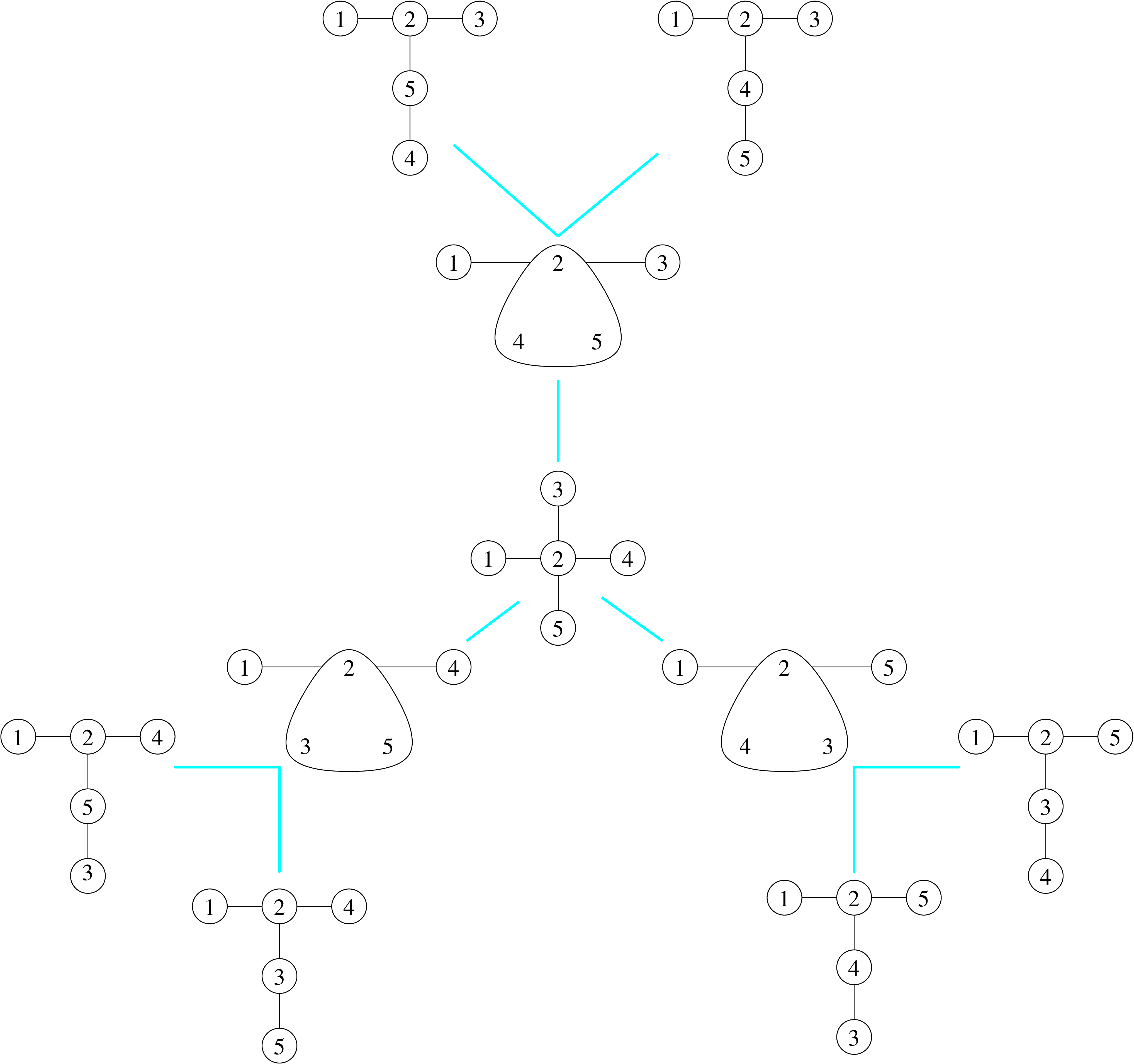}
 \caption{The peripheral subcomplex of $|\hyper_5|$ associated to the
one--two outer-automorphism $\balpha_{\{3,4,5\}, 2}$}
 \label{fig:peripheral}
\end{center}
\end{figure}

\begin{lem}\label{lem:essentialmeansnoperiph}
Let $\A$ be a compatible collection of one--two automorphisms and let
$\tau(\A)$ be the cone point of $\A$.  Then $\A$ is essential if and
only if the $\A$--peripheral complex is empty.
\end{lem}

\begin{proof}
Let $\A$ be essential and let $\alpha_{Ij} \in \A$.  The sub-hypertree
induced by $I \cup \{j\}$ is an ordinary tree, with the vertex
corresponding to $j$ being a leaf.  Since every hypertree $\tau$ with
$\tau(\A) < \tau$ is formed by dividing the single fat edge of
$\tau(\A)$ into a higher rank hypertree, the sub-hypertree of $\tau$
induced by $I \cup \{j\}$ is the same ordinary tree.  Hence
$\alpha_{Ij} \in \B(\tau)$, so $\A \subset \B(\tau)$.  Thus by
definition the $\A$--peripheral complex is empty.

If $\A$ is inessential, there must be a one--two automorphism
$\alpha_{Ij}$ where the sub-hypertree induced by $I \cup \{j\}$
contains a fat edge, $e$.  Let $k$ label the vertex of $e$ that is
closest to $j$.  Then splitting $e$ into two hyperedges that are
joined along $k$ creates a hypertree $\tau$ that is above $\tau(\A)$
in $\hyper_n$ but $\alpha_{Ij} \not \in \B(\tau)$.  Thus $\tau$ is
peripheral, and the $\A$--peripheral complex is not empty.
\end{proof}

\section{Computing the $E_2$ page}\label{sec:Etwo}

In this section we establish:

\begin{prop}\label{prop:collapse}
The non-zero entries on the $E_2$ page for the spectral sequence
corresponding to the action $\opsig_n \curvearrowright \MM_n$ are
concentrated in the $(0,q)$--column.
\end{prop}

To start the process of proving this, we note that the one--two bases
for the stabilizers of hypertrees allow us to give a concrete
description of the entries on the first page of the spectral sequence.
The $(p,q)$ entry on the $E_1$ page is the product
\[
\prod_{\begin{array}{c}{\sigma \in {|\hyper_n|},}\\ 
{\mbox{dim}(\sigma) = p}\end{array}} H^q (\Stab(\sigma), \Z)\ .
\]   
Since the stabilizer of $\sigma$ is the stabilizer of its minimal
element, we may consider what happens in the case of a hypertree
$\tau$.  Each stabilizer is free abelian, so the first homology
$H_1(\mbox{Stab}(\tau), \Z)$ is free abelian with generating set
$\{[\balpha_{Ij}]~|~\balpha_{Ij} \in \B(\tau)\}$.  The cohomology
group $H^1(\mbox{Stab}(\tau), \Z)$ is then generated by the dual basis
$\{\balpha^*_{Ij}\}$ where
\[
\balpha_{Ij}^*([\balpha_{Kl}]) = \left\{\begin{array}{ll}
1 & I = K\mbox{ and } j = l\\ 
0 & \mbox{otherwise}.\\ \end{array}
\right.
\]
Because the $H^*(\Z^n)$ is an exterior algebra generated by
one-dimensional classes, the set $\{\balpha_{Ij}^*~|~\balpha_{Ij} \in
\B(\tau)\}$ is a generating set for the cohomology of the stabilizer.
It follows that the set of all products of $q$ distinct
$\balpha_{Ij}$s is a generating set for $H^q(\Stab(\tau),\Z)$.  Thus
we may identify the generating set for $H^q(\Stab(\sigma),\Z)$ with
the collection of all subsets of one--two outer-automorphisms $\A
\subset \B(\tau)$ with $|\A| = q$.

As the differentials are a combination of restrictions to direct
summands and coboundary maps, the $E_1$ page of the equivariant
spectral sequence can be expressed as a union of sub-cochain
complexes.  Let $\A$ be a compatible collection of one--two
automorphisms, with $|\A|=q$.  Then in row $q$ one sees the terms
where $\A$ describes a generator, corresponding to simplices $\sigma$
with $\B(\sigma) \supset \A$.

\begin{lem}\label{lem:essential}
If $\A$ is essential, then the collection of all entries in row $q$ of
the first page that are given by having $\A \subset \B(\tau)$ forms a
sub-cochain complex whose cohomology consists of a single $\Z$ in
dimension zero.
\end{lem}

\begin{proof}
Since $\A$ is essential, its support is a subcomplex of the hypertree
poset.  That is, if $\sigma \in |\hyper_n|$ is in the support of $\A$,
then every face of $\sigma$ is in the support of $\A$.  Further, the
vertex associated to $\tau(\A)$ is a cone point for the support of
$\A$, hence the support of $\A$ is contractible.  The fact that the
support is actually a subcomplex shows that the sub-cochain complex in
the $E_1$ page that corresponds to $\A$ is just the cochain complex
for the support of $\A$.  Because the support is contractible, the
cohomology of this cochain complex is trivial except for a single $\Z$
in dimension zero.
\end{proof}

We now turn to the case where $\A$ is inessential.  Call a fat edge in
$\tau(\A)$ \emph{worrisome} if it does not contain the vertex labelled
$1$ or if it does contain the vertex labelled $1$, but the reduced
path joining $1$ to $2$ does not contain this edge.  Let $\{e_1,
\ldots , e_k\}$ be the worrisome fat edges of $\tau(\A)$.  Let $c_i$
be the label of the vertex in $e_i$ that is closest to the vertex
labelled $1$ and let $L_i$ be the remaining labels of vertices in
$e_i$.  Thus, viewing $\tau(\A)$ as a hypertree rooted at $1$, the
vertices labelled by numbers in $L_i$ are one level down from the
vertex $c_i$.

If ${\mathcal P}_i$ is a partition of $L_i$ define
$\tau(\A)_{{\mathcal P}_i}$ to be the hypertree where the hyperedge
$e_i$ has been split into hyperedges defined by this partition.  That
is, if ${\mathcal P_i} = \{\ell_1, \ldots , \ell_m\}$ then the edge
$e_i$ will be replaced with the edges $\{\ell_j \cup \{c_i\}~|~1 \le j
\le m\}$.  Moreover, if ${\mathfrak p} \in \Pi_1 \times \cdots \times
\Pi_k$ is a collection of partitions of the $L_i$, define
$\tau(\A)_{{\mathfrak p}}$ to be the hypertree formed by splitting
each worrisome edge $e_i$ by the associated partition of $L_i$.

\begin{exmp}
Consider the hypertree $\tau$ shown in \fullref{fig:gravity}.  The
worrisome edges are all of the edges of $\tau$, excepting $e = \{1, 2,
3\}$.  Set $L_1 = \{4,5\}, L_2 = \{6, 7\}, L_3 = \{8,9\}$ and $L_4 =
\{10, 11\}$ in order to fix the indexing.  Let
\[
\mathfrak{p} = \left(
\{\{4\},\{5\}\}, \{\{6, 7\}\},  \{\{8, 9\}\}, \{\{10\},\{11\}\}
\right)
\]
be the result of taking non-trivial partitions of $L_1$ and $L_4$.
Then the tree $\tau_{{\mathfrak p}}$ formed by splitting $\tau$
according to this product of partitions is shown in
\fullref{fig:gravitysplit}.
\end{exmp}

\begin{figure}[ht!]
\begin{center}
\includegraphics[width=12cm]{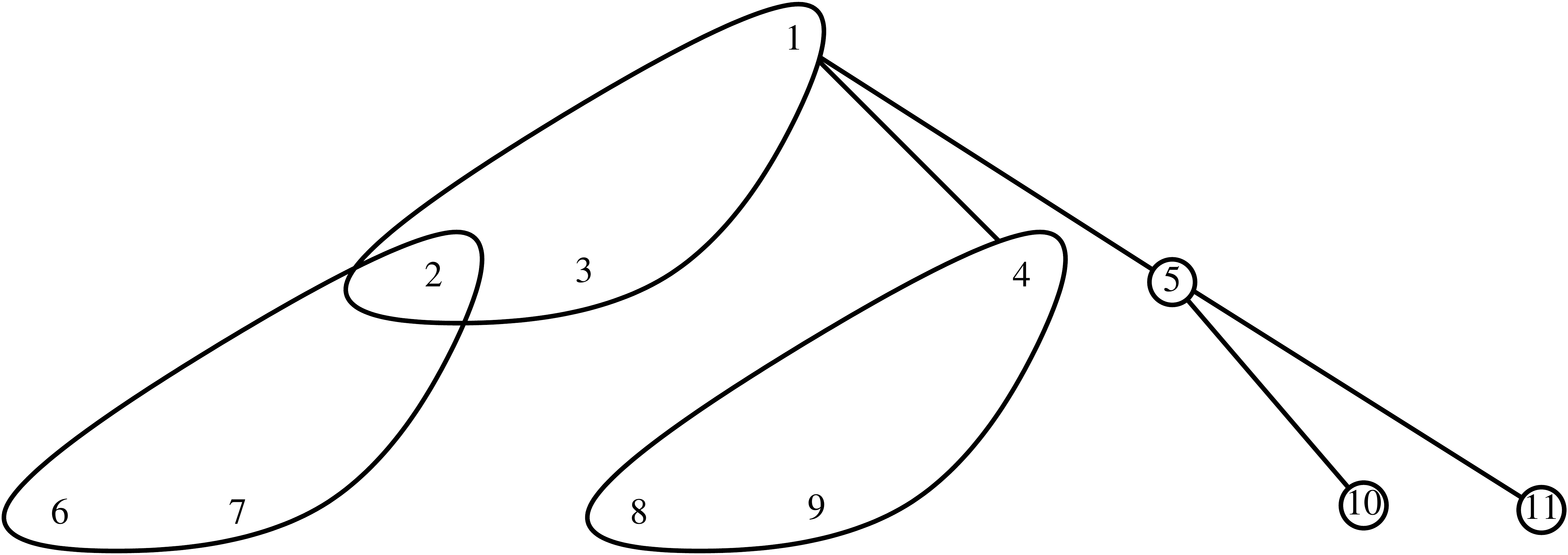}
\caption{The hypertree resulting from splitting some of the worrisome
fat edges in the hypertree shown in \fullref{fig:gravity}}
\label{fig:gravitysplit}
\end{center}
\end{figure}

\begin{lem}\label{lem:contractible}
Let $\A$ be an inessential collection of one--two automorphisms.  Then
the $\A$--peripheral complex is contractible.
\end{lem}

\begin{proof}
Let $\mathfrak p$ be a non-trivial element of $\Pi_1 \times \cdots
\times \Pi_k$.  That is, assume at least one of the partitions is
non-trivial.  Let $\hyper_{\ge \mathfrak p}$ be the subposet of
$\hyper_n$ consisting of elements greater than or equal to
$\tau(\A)_{{\mathfrak p}}$.  The associated order complexes
$|\hyper_{\ge \mathfrak p}|$ cover the peripheral complex.

Each $|\hyper_{\ge \mathfrak p}|$ is contractible as the vertex
associated to $\tau(\A)_{\mathfrak p}$ forms a cone point.  Further,
if ${\mathfrak{p}_1 , \ldots , \mathfrak{p}_s}$ is any collection of
non-trivial elements in $\Pi_1 \times \cdots \times \Pi_k$ then their
meet $ \mathfrak{p} = {\mathfrak{p}_1 \wedge \cdots \wedge
\mathfrak{p}_s}$ is a non-trivial element of $\Pi_1 \times \cdots
\times \Pi_k$ and
\[
|\hyper_{\ge \mathfrak{p}}| = |\hyper_{\ge \mathfrak{p}_1}|
\cap \cdots \cap  |\hyper_{\ge \mathfrak{p}_s}|\ .
\]
Thus we have covered the peripheral complex by contractible
subcomplexes whose intersections are also contractible.  By the
Quillen Fiber Lemma, the peripheral complex is homotopy equivalent to
the nerve of this covering.

Let $\hat{\mathfrak{p}}$ be the product of partitions given by totally
partitioning each $L_i$.  Then $|\hyper_{\ge \hat{\mathfrak{p}}}|$ is
a single vertex that is a cone point in the nerve of the covering.
Thus the nerve of the covering is contractible, hence so is the
$\A$--peripheral complex.
\end{proof}

The argument above may become less opaque if one consults the example
of a peripheral complex given in \fullref{fig:peripheral}.  The
original $\tau(\A)$ is shown in \fullref{fig:examplecone}, and the
single worrisome edge is $e = \{2, 3, 4, 5\}$.  The reader will find
that this peripheral complex has been covered by four contractible
subcomplexes corresponding to the four non-trivial partitions of $\{3,
4, 5\}$.

\begin{lem}\label{lem:inessential}
If $\A$ is inessential, then the collection of all entries in row $q$
of the first page that are given by having $\A \subset \B(\sigma)$
forms a sub-cochain complex whose cohomology is trivial in all
dimensions.
\end{lem}

\begin{proof}
Let $\A \subset \B(\sigma)$.  Then $\sigma$ corresponds to a chain of
hypertrees $\tau_0 < \cdots < \tau_p$ and $\A \subset \B(\tau_0)$.
Since $\tau(\A) \le \tau_0$ we can pair $p$--simplices where $\tau_0
\ne \tau(A)$ with $(p+1)$--simplices formed by adding $\tau(\A)$ to the
chain:
\[
\underbrace{\tau_0 < \cdots < \tau_p}_{\sim \sigma} \leftrightarrow
 \underbrace{\tau(\A) < \tau_0 < \cdots < \tau_p}_{\sim \sigma'}\ .
\]
We can remove the terms corresponding to paired simplices in the
original cochain complex to form a cochain complex with equivalent
cohomology; in other words we may restrict ourselves to cochains with support
on chains with initial element $\tau(\A)$.  Thus the terms in this new cochain complex correspond to
chains of hypertrees of the form $\tau(\A) < \tau_1 < \cdots < \tau_k$
where each $\tau_i$ (for $i \ge 1$) is in the $\A$--peripheral
subcomplex of $|\hyper_n|$.  In dimension zero there is a single
$\Z$ corresponding to $\tau(\A)$.  In general, in dimension $k$ there
is a $\Z^{p(k)}$ if $p(k)$ is the number of $(k-1)$--simplices in the
$\A$--peripheral subcomplex.  Thus the new cochain complex is simply
the augmented cochain complex for the $\A$--peripheral complex (with
indices shifted by one).  Since the $\A$--peripheral complex is
contractible (\fullref{lem:contractible}) its reduced cohomology is
trivial hence this cochain complex is acyclic.
\end{proof}

\section{Computing ranks via planted forests}

We now know that the $E_2$ page consists of a single column and the
entries are $\Z^i$ where $i$ counts the number of essential sets of
one--two basis elements.  By its definition, every essential set of
compatible one--two basis elements forms the basis of a hypertree with
at most one edge of order $>2$.  These hypertrees can be described
using planted forests.

Given a planted forest on $[n]$, $f$, let $\tau_{f}$ be the hypertree
whose edges consist of the edges of $f$ and one additional (hyper)edge
consisting of the roots of $f$.  That is, one forms the hypertree
$\tau_f$ by gathering the roots of $f$.  (See
\fullref{fig:gathering}.)

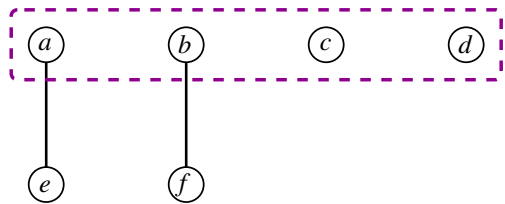
\begin{figure}[ht!]\small
\begin{center}
\labellist
\pinlabel $a$ at 38 164  
\pinlabel $b$ at 182 164  
\pinlabel $c$ at 327 164  
\pinlabel $d$ at 471 164  
\pinlabel $e$ at 38 18  
\pinlabel $f$ at 182 19  
\endlabellist
\includegraphics[width=2.6in]{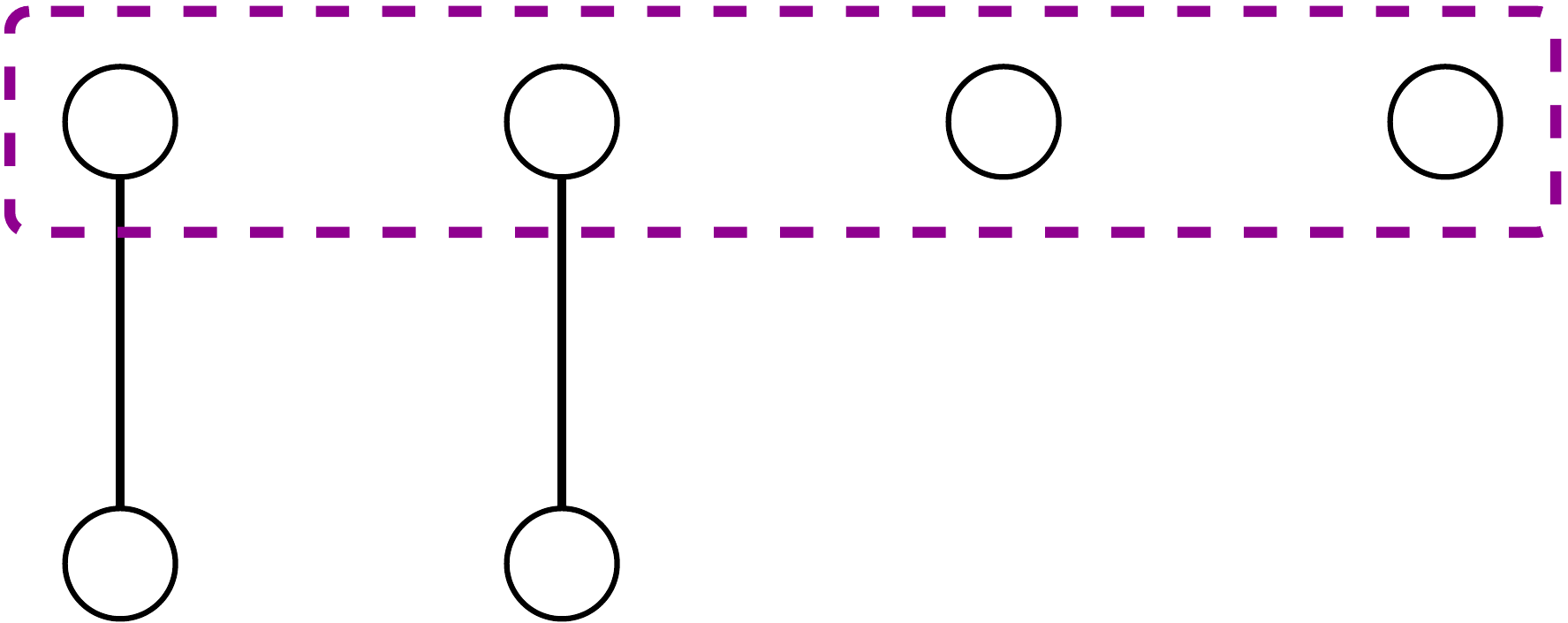}
 \caption{Gathering a planted forest with four components on $[6]$ into
a hypertree}
 \label{fig:gathering}
\end{center}
\end{figure}

\begin{lem}\label{lem:gatheringdescription}
A rank $q$ hypertree $\tau$ is \emph{essential} if and only if:
\begin{enumerate}
\item $\tau$ is formed by gathering the roots of a planted forest $f$
on $[n]$ with (n-q) components; and
\item The vertex labelled $1$ is a root of $f$ and it is not the root
of the tree containing the vertex labelled $2$.
\end{enumerate}
\end{lem}

There is a bijection between essential sets of one--two generators with
$q$ elements and rank $q$ essential hypertrees.  Thus it suffices to
count the rank $q$ essential hypertrees.  The number of $(n-q)$
component planted forests on $[n]$ is
\[
p_{n-q}(n) = \binom{n-1}{q} n\sp{q}
\]
(See  Proposition 5.3.2 of \cite{StVol2}.)

Because of our second condition in
\fullref{lem:gatheringdescription}---on the location of the vertex
labelled by $1$---not all hypertrees formed by gathering forests are
essential.  If any of the non-root vertices are labelled $1$ we would
be unhappy.  Similarly, referring to \fullref{fig:gathering}, if
$a=1$ then we would not want $e=2$.  That is, $2$ should not label a
vertex in the tree rooted by $1$.

So given an unlabelled planted forest with $k$ components and $n$
vertices, what fraction of the labellings lead to essential
hypertrees?  First, place $2$.  There are then $(n-1)$ places you
could put $1$.  Of these only the $(n-q-1)$--roots that are not the
roots of the tree containing $2$ will lead to essential hypertrees.
Thus the number of essential hypertrees can be gotten by taking all
hypertrees and multiplying by $\frac{n-q-1}{n-1}$.

\begin{lem}\label{lem:rank}
The number of essential hypertrees in $\hyper_n$ of rank $q$   is
\[
\binom{n-2}{q} n\sp{q}\ .
\]
Therefore $\rk\left[H^q(\opsig_n,\Z) \right]= \binom{n-2}{q} n\sp{q}$.
\end{lem}

When $q$ is maximal (ie $q=n-2$) the hypertrees under consideration
aer ordinary trees, all of them are essential, and
\fullref{lem:rank} reproves the standard count of $n^{n-2}$ for the
number of trees with vertex set $[n]$.

\fullref{lem:rank} establishes our Main Theorem.  In order to prove
the Brownstein--Lee Conjecture we need a bit more information.  It
follows from McCool's presentation that $H_1(\psig_n, \Z)$ is free
abelian with one generator $[\alpha_{ij}]$ for each generator of
$\psig_n$.  The cohomology group $H^1(\psig_n, \Z)$ is then generated
by the dual basis $\{\alpha^*_{ij}\}$ where
\[
\alpha_{ij}^*([\alpha_{kl}]) = \left\{\begin{array}{ll}
1 & (k, l) = (i, j)\\ 
0 & \mbox{otherwise}.\\ \end{array}
\right.
\]
Similarly $\opsig_n$ is generated by $\{\balpha_{ij}~|~i \ne j, i \ne
1, \mbox{ and }i \ne 2 \mbox{ if } j=1\}$, its first homology is
generated by the associated $[\balpha_{ij}]$ and $H^1(\opsig_n, \Z)$
is generated by the dual basis $\alpha_{ij}^*$.  
The argument in \fullref{sec:Etwo}, establishing that the $E_2$ page
is concentrated in a single column (\fullref{prop:collapse}), gives 
an explicit description of the groups $E_2^{0,q}$.  Namely, they are free abelian
where the elements in our chosen generating set correspond to
products of the elements $\displaystyle \alpha_{I,j}^* = \sum_{i \in I} \alpha_{ij}^*$.
These cohomology classes $\alpha_{I,j}^*$ come from a generating set
for the stablizer of an  essential hypertree, hence
 it must be the case that for any $i \in I$, $i \not = 1$ and
$i \not = 2$ if $j=1$.  Thus we get:

\begin{prop}\label{prop:outproductstruct} 
The algebra $H^*(\opsig_n, \Z)$
is generated by the one-dimensional classes $\balpha_{ij}^*$ where $i \not = j$, 
$i \not = 1$ and
$i \not = 2$ if $j=1$.
\end{prop}

\section{Computing $H^*(\mbox{P}\Sigma_n,\Z)$}
\label{sec:cohopsig}

Starting with the short exact sequence
\[
1\rightarrow \F_n \rightarrow \psig_n \rightarrow \opsig_n \rightarrow 1
\]
one can apply the Lyndon--Hochschild--Serre spectral sequence to compute
the cohomology groups $H^i(\psig_n,M)$.  Since $\mbox{cd}(\F_n) = 1$
this spectral sequence is concentrated in the bottom two rows of the
first quadrant (see \fullref{fig:sslhs}).  In our case, it is
easiest to understand the structure of this spectral sequence by
thinking in terms of a fibration, and appealing to the Leray--Hirsch
Theorem.  There is a fiber bundle $p: B\psig_n \rightarrow B\opsig_n$
with fiber $B\F_n$ which gives a fibration (see Theorem 1.6.11 and
Theorem 2.4.12 of \cite{Be91}).  The action of $\psig_n$ on $\F_n$
always sends a generator to a conjugate of itself and so the induced
action on $H^1(\F_n, \Z)= \Z^n$ is trivial; therefore $E_2^{*,1} =
H^*(\opsig_n, H^1(\F_n, \Z)) = H^*(\opsig_n, \Z^n)$ with the action of
$\opsig_n$ on $\Z^n$ being trivial.  Hence the system of local
coefficients in the spectral sequence is simple.

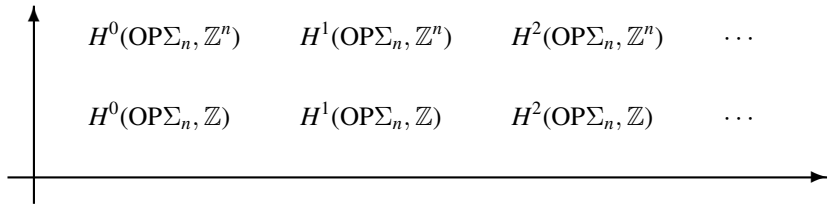
\begin{figure}[ht!]\small
\begin{center}
\begin{picture}(320,80)(-10, -10)
\thicklines
\put(-10,0){\vector(1,0){310}}
\put(0,-10){\vector(0,1){75}}
%
\put(20,20){$\displaystyle H^0(\opsig_n, \Z)$}
\put(20,50){$\displaystyle 
	H^0(\opsig_n, \Z^n)$}

\put(100,20){$\displaystyle H^1(\opsig_n, \Z)$}
\put(100,50){$\displaystyle 
	H^1(\opsig_n, \Z^n)$}

\put(180,20){$\displaystyle H^2(\opsig_n, \Z)$}
\put(180,50){$\displaystyle 
	H^2(\opsig_n, \Z^n)$}

\put(260,20){$\cdots$}
\put(260,50){$\cdots$}

\end{picture}
\end{center}
 \caption{The $E_2$ page of the Lyndon--Hochschild--Serre
 spectral sequence for $1 \rightarrow \F_n \rightarrow \psig_n
\rightarrow \opsig_n \rightarrow 1$}
 \label{fig:sslhs}
 \end{figure}

We remind the reader of the Leray--Hirsch Theorem, as presented in \cite{Ha02}:

\begin{thm}[Leray--Hirsch Theorem]\label{thm:leray-hirsch}
Let $F \stackrel{\iota}{\rightarrow} E \stackrel{\rho}{\rightarrow} B$
be a fiber bundle such that
\begin{enumerate}
\item $H^n(F,\Z)$ is a finitely generated free $\Z$--module for each
$n$, and
\item there exist classes $c_j \in H^{k_j}(E, \Z)$ whose restrictions
$\iota^*(c_j)$ form a basis for $H^*(F, \Z)$ in each fiber $F$.
\end{enumerate}
Then the map $\Phi: H^*(B, \Z) \otimes H^*(F, \Z) \rightarrow H^*(E,
\Z)$ given by
\[
\sum_{i,j} b_i \otimes \iota^*(c_{j}) \mapsto \sum_{i,j} \rho^*(b_i) \cup c_{j}
\]
is an isomorphism.
\end{thm}

The first condition of the Leray--Hirsch Theorem is immediately
satisfied since the kernel we are interested in is a free group.  We
turn then to establishing that the fiber $B\F_n$ is totally
non-homologous to zero in $B\psig_n$ (with respect to $\Z$).

\begin{lem}\label{lem:totallynonhomologoustozero}
The space $B\F_n$ is totally non-homologous to zero in $B\psig_n$ with
respect to $\Z$.  In fact, let $c_{0,0}$ be a generator for
$H^0(B\psig_n, \Z)$ and let $c_{1,1}, \ldots, c_{1,n}$ denote the
duals in $H^1(B\psig_n, \Z) = Hom((B\psig_n)_{ab}, \Z)$ of maps
corresponding to conjugating by the generators of $\F_n$.  Then the
collection $\{\iota^*(c_{r,s})\}$ forms an additive basis for
$H^*(B\F_n, \Z)$, where $\iota: B\F_n \to B\psig_n$ is inclusion.
\end{lem}

\begin{proof}
The inner automorphism of $\F_n$ given by conjugating by a basis
element $x_j$ is sent to the symmetric automorphism
\[
\alpha_{[n]\setminus \{j\}, j} = \alpha_{1j} \cdot \alpha_{2j} \cdots
\hat \alpha_{jj} \cdots \alpha_{nj} 
\]
under the injection $\F_n \hookrightarrow \psig_n$.  As the
$\alpha_{ij}$ form a generating set for $\psig_n$ that projects to a
minimal generating set for $H_1(\psig_n, \Z)$, the map
\[
(\F_n)_{ab} \rightarrow (\psig_n)_{ab}
\]
is injective.  Hence the map
\[
H^1(\psig_n, \Z) = \mbox{Hom}((\psig_n)_{ab}, \Z) \rightarrow
H^1(\F_n, \Z) = \mbox{Hom}((\F_n)_{ab}, \Z)
\]
is onto.  Since the cohomology of $\F_n$ is only located in degrees 0
and 1, this means that the map $H^*(\psig_n, \Z) \rightarrow H^*(\F_n,
\Z)$ is onto.
\end{proof}

Having satisfied the hypotheses we may apply the Leray--Hirsch Theorem
to obtain

\begin{lem}\label{lem:lerayhirsch}
The map 
\[
 H^*(\opsig_n, \Z) \otimes H^*(\F_n, \Z) \to H^*(\psig_n, \Z)
\]
defined by 
\[
\sum_{j,s} b_j \otimes i^*(c_{r,s}) \mapsto \sum_{j,s} p^*(b_j) \cup c_{r,s}
\]
is an isomorphism.  Furthermore, the above spectral sequence has
trivial differential $d_2$, and therefore the rank of $H^i(\psig_n,
\Z)$ is $\binom{n-1}{i} \cdot n^i$
\end{lem}

\begin{proof}
That the map is an isomorphism follows from the Leray--Hirsch Theorem.
To compute the ranks of the cohomology groups, we note that
$$H^i(\psig_n, \Z) = H^{i-1}(\opsig_n, \Z^n) \times H^i(\opsig_n, \Z)$$
which by the Main Theorem means
\[
H^i(\psig_n, \Z) = \Z^{\binom{n-2}{i-1} \cdot n^{i-1} \cdot n}
 		\times \Z^{\binom{n-2}{i} \cdot n^i}
= \Z^{\left[\binom{n-2}{i-1}+ \binom{n-2}{i}\right]\cdot n^i}
= \Z^{\binom{n-1}{i} \cdot n^i}\,. \proved
\]
\end{proof}

Thus at the level of abelian groups, we have the formula claimed in
\fullref{thm:productstruct} below.  At this point the only cause
for caution is that the Leray--Hirsch Theorem does not immediately
imply a \emph{ring} isomorphism.  This we establish via the next two
results.

\begin{cor}\label{cor:onedimintegral} 
The integral cohomology of $\psig_n$ is generated by one-dimensional
classes.
\end{cor}

\begin{proof}
Consider an $x = p^*(b) \cup c_{1,s} \in H^*(\psig_n, \Z)$.  From
\ref{prop:outproductstruct}, $H^*(\opsig_n, \Z)$ is generated by
one-dimensional classes and so there are $d_1, \ldots, d_t \in
H^1(\opsig_n, \Z)$ such that $b = d_1 \cup \cdots \cup d_t$.  Hence
$p^*(b) = p^*(d_1 \cup \cdots \cup d_t) = p^*(d_1) \cup \cdots \cup
p^*(d_n)$ (cf property (3.7) in Chapter V of \cite{Br94}) is also a
product of one-dimensional classes.  So $x$ is a product of
one-dimensional classes.
\end{proof}

In their work on $H^*(\psig_n, \Z)$, Brownstein and Lee established the
following:

\begin{thm}[Theorems 2.10 and 2.11 in \cite{BrLe93}]\label{thm:BrLeeSummary}
The relations
\begin{enumerate}
\item $\alpha_{ij}^* \wedge \alpha_{ij}^* = 0$
\item $\alpha_{ij}^* \wedge \alpha_{ji}^* = 0$
\item $\alpha_{kj}^* \wedge \alpha_{ji}^* = (\alpha_{kj}^* - \alpha_{ij}^*) \wedge \alpha_{ki}^*$
\end{enumerate}
hold among the one-dimensional classes $\alpha_{ij}^*$ in $H^*(\psig_n, \Z)$, and---in the
case where $n=3$---they  present the  algebra $H^*(\psig_3, \Z)$.
\end{thm}

These relations give an upper bound on the 
rank of  $H^i(\psig_n, \Z)$.   

\begin{lem}\label{lem:boundranks}
The Brownstein--Lee relations imply that the rank of $H^i(\psig_n, \Z)$
is at most $\binom{n-1}{i} n^i$.
\end{lem}

\begin{proof}
The first relations combined with the fact that one-dimensional
classes generate $H^*(\psig_n, \Z)$ imply that we may form a basis for
$H^i(\psig_n, \Z)$ using $i$--fold products of the $\alpha_{ij}^*$ with
no repetitions.  We may rewrite the third set of relations as
\[
\alpha_{kj}^* \alpha_{ki}^* = \alpha_{kj}^*  \alpha_{ji}^* +
\alpha_{ij}^*\alpha_{ki}^* 
\]
implying that we never need to repeat the first index.  An induction
argument using the original expression for the third relation shows
that any telescoping sequence $\alpha_{ij}^* \alpha_{jk}^*
\alpha_{kl}^* \cdots \alpha_{st}^*$ is equivalent to a linear
combination of terms, each of which includes $\alpha_{it}^*$.  But
then applying the second relation we see that any cyclic product is
trivial:
\[
\alpha_{ij}^* \alpha_{jk}^* \alpha_{kl}^* \cdots \alpha_{st}^*
\alpha_{ti}^* = \left( \sum \pm [\mbox{various }(i-1)\mbox{--fold
products}] \alpha_{it}^*\right) \alpha_{ti}^* = 0.
\]
Thus $H^i(\psig_n, \Z)$ is generated as an abelian group by the
$i$--fold products of the one-dimensional generators that do not repeat
any $\alpha_{ij}^*$, do not repeat a first index, and which do not
contain any cyclic products.

We may encode such products as directed graphs on $[n]$ where
including $\alpha_{jk}^*$ in the product adds an edge ``$j \leftarrow
k$".  The fact that no cyclic products are allowed implies the graph
has no cycles, so it is a forest.  The fact that no first index is
repeated implies that each vertex is the target of at most one edge,
hence there is a natural planting of the forest, with the roots
corresponding to vertices that are not targets of any directed edge.
Since there are $i$ edges, and $n$ vertices, the forest must have
$(n-i)$ connected components.  But the number of $k$--component forests
on $[n]$ is $\binom{n-1}{k-1}n^{n-k}$ hence the number of
$(n-i)$--component forests on $[n]$---which gives the upper bound on
the rank of $H^i(\psig_n, \Z)$---is
\[
\binom{n-1}{(n-i)-1}n^{n-(n-i)} = \binom{n-1}{i}n^i . \proved
\]
\end{proof}

We know that $H^*(\psig_n, \Z)$ is a quotient of the algebra presented
by Brownstein and Lee's relations, but we also know by
\fullref{lem:lerayhirsch} that the rank of $H^i(\psig_n, \Z)$ is
$\binom{n-1}{i}n^i$.  Thus we have established the Brownstein--Lee
Conjecture (the Corollary to our Main Theorem).

\begin{thm}\label{thm:productstruct} 
The cohomology of $H^*(\psig_n, \Z)$ is generated by one-dimensional
classes $\alpha_{ij}^*$ where $i \not = j$, subject to the relations
\begin{enumerate}
\item $\alpha_{ij}^* \wedge \alpha_{ij}^* = 0$
\item $\alpha_{ij}^* \wedge \alpha_{ji}^* = 0$
\item $\alpha_{kj}^* \wedge \alpha_{ji}^* = (\alpha_{kj}^* -
  \alpha_{ij}^*) \wedge \alpha_{ki}^*$.
\end{enumerate}
In particular, the Poincar\'e series is ${\mathfrak p}(z) = (1 +
nz)^{n-1}$.
\end{thm}

\begin{exmp}
The cohomology groups of $\psig_4$ are 
\[
{H}^i(\psig_{4}, \Z) = 
 \left\{ 
\begin{array}{lc}
\Z^{64} & i = 3\\
\Z^{48} & i = 2\\
\Z^{12} & i = 1\\
\Z & i=0\,.
\end{array}
\right.
\]
\end{exmp}

As was remarked in \fullref{sec:background}, viewed as a subgroup of
$\mbox{Aut}(\F_n)$, $\psig_n$ is contained in the subgroup
$\mbox{IA}_n$.  Magnus proved that $\mbox{IA}_n$ is generated by the
$\alpha_{ij}$, which generate $\psig_n$, along with the automorphisms
induced by
\[
\theta_{ijk}= \left\{
\begin{array}{lc}
x_i \rightarrow x_i[ x_j , x_k] & \cr
x_l \rightarrow x_l & l \ne i\,.\cr
\end{array}  \right.
\]
Aside from this generating set,  little is known about
$\mbox{IA}_n$.  Fred Cohen and Jon Pakianathan, and independently
Benson Farb, showed that $H_1(\mbox{IA}_n)$ is free abelian of rank
$n^2(n-1)/2$, and is generated by the classes $[\alpha_{ij}]$ and
$[\theta_{ijk}]$.  Krsti\'c and McCool proved that $\mbox{IA}_3$ is
not finitely presentable \cite{KrMcC97}.  The Krsti\'c--McCool result
has recently been
extended by Bestvina, Bux and Margalit to $\mbox{IA}_n$ for $n \ge 3$; they
also succeed in computing the cohomological dimension of $\mbox{IA}_n$
and in showing that its top dimensional cohomology is not finitely generated \cite{BeBuMa06}. 
Fred Cohen has pointed out to us that
\fullref{thm:productstruct} gives some information about
$H^*(\mbox{IA}_n, \Z)$:

\begin{cor}\label{cor:cohencorollary}
The injection $\psig_n \hookrightarrow \mbox{IA}_n$ induces a split
epimorphism 
\[
H^*(\mbox{IA}_n, \Z) \twoheadrightarrow H^*(\psig_n, \Z)\ .
\]
Moreover, the suspension of $B\psig_n$ is homotopy equivalent to a
bouquet of spheres, and it is a retract of the suspension of
$B\mbox{IA}_n$.
\end{cor}

\begin{rem}
The suspension of $BP_n$, where $P_n$ is the pure braid group, is also
homotopy equivalent to a wedge of spheres. (This is Corollary 3 of
\cite{Sch97}; see also \cite{AdCoCo03}, where this property and its
implications for $K$--theory are explored.)  Like the inclusion
$\psig_n \hookrightarrow \mbox{IA}_n$, the inclusion $P_n
\hookrightarrow \psig_n$ induces a surjection on cohomology:
$H^*(\psig_n, \Z) \twoheadrightarrow H^*(P_n, \Z)$, but this
surjection does not split over $\Z$, but it does split over
$\Z[\frac{1}{2}]$ (Proposition~4.3 of \cite{BrLe93}).
\end{rem}

\begin{proof} [Proof of \fullref{cor:cohencorollary}, due to Fred Cohen]
Consider the map $\psig_n \hookrightarrow \mbox{IA}_n$ composed with
the abelianization
 \[
 \mbox{IA}_n \to \bigoplus_{n \binom{n}{2}} \Z = H_1(\mbox{IA}_n, \Z)\ .
 \]
The induced map from the first homology of $\psig_n$ to the first
homology group of $\Z^{{n \binom{n}{2}}}$ is a split
monomorphism. Thus the induced map on the level of integral cohomology
is a split epimorphism, at least in degree 1. Since $H^1(\psig_n, \Z)$
generates $H^*(\psig_n, \Z)$ as an algebra, it follows that
\[
H^*(\bigoplus_{n \binom{n}{2}} \Z, \Z) \to H^*(\psig_n, \Z)
\]
is an epimorphism.

However, the cohomology of $\Z^{n \binom{n}{2}}$ is an exterior
algebra on one-dimensional classes, some of which correspond to the
generators for the cohomology of $\psig_n$.  Furthermore, the single
suspension of $B[\Z^{n \binom{n}{2}}]$ is a bouquet of spheres, with
some subset of these corresponding exactly to the image of homology of
$\psig_n$ (suitably reindexed).  Denote this subset by
$S_{\mbox{\scriptsize P}\Sigma_n}$.  We can then project from the
suspension of $B[\Z^{n \binom{n}{2}}]$,
\[
\Sigma( B[\Z^{n \binom{n}{2}}]) \rightarrow S_{\mbox{\scriptsize P}\Sigma_n}\ .
\]
The composition
\[
\Sigma( B\psig_n) \rightarrow \Sigma( B[\Z^{n \binom{n}{2}}]) \rightarrow 
S_{\mbox{\scriptsize P}\Sigma_n}
\]
induces a homology isomorphism (of simply connected spaces) and thus
is a homotopy equivalence.  It follows that $\Sigma(B\psig_n)$ is a
retract of $\Sigma(B\mbox{IA}_n)$.
\end{proof}

\begin{rem}  
Do the connections above give new information about the cohomology of
$\mbox{IA}_n$?  Using the Johnson homomorphism Alexandra Pettet has
found a large number of cohomology classes in $H^2(\mbox{IA}_n, \Z)$
\cite{Pe05}, all of which are in the image of the map $\wedge^2
H^1(\mbox{IA}_n, \Z) \rightarrow H^2(\mbox{IA}_n, \Z)$.  Pettet points
out that as a consequence of \fullref{cor:cohencorollary} one
gets a commutative diagram
\[
\begin{CD}
\wedge^2 H^1(\mbox{IA}_n, \Z) @>>> \wedge^2 H^1(\psig_n, \Z)\\
@VVV                                                          @VVV\\
H^2(\mbox{IA}_n, \Z) @>>>  H^2(\psig_n, \Z)
\end{CD}
\]
where all the maps, excluding the left edge, are surjections.  Thus
the classes arising from \fullref{cor:cohencorollary} are
included in the classes Pettet has found.

Another curious point is that \fullref{cor:cohencorollary} gives
a non-trivial map
\[
H_*(BIA_n) \to T[\bar H_*(B\psig_n)]
\]
where T[V] denotes the tensor algebra generated by $V$.  This process
also gives a map out of $B\mbox{IA}_n$ to a highly non-trivial space,
which \emph{might} give more information about $B\mbox{IA}_n$. The map
is
\[
 B\mbox{IA}_n \rightarrow \Omega \Sigma( B\psig_n)
\]
factoring the Freudenthal suspension $B\psig_n \rightarrow \Omega
\Sigma (B\psig_n)$.  By arguments similar to those above, the
composite
\[
 B\psig_n \rightarrow B\mbox{IA}_n \rightarrow \Omega \Sigma( B\psig_n )
\]
induces a split epimorphism in cohomology.  For information on such
arguments, in particular how they apply to the subgroup of $\psig_n$
generated by $\{\alpha_{ij}~|~i > j\}$, see \cite{CoPa05}.
\end{rem}

\bibliographystyle{gtart}
\bibliography{link}

\end{document}